\documentclass[10pt,bezier]{article}
\usepackage{amsmath,amssymb,amsfonts,graphicx,caption,subcaption}
\usepackage[colorlinks,linkcolor=red,citecolor=blue]{hyperref}

\newtheorem{prethm}{{\bf Theorem}}[section]
\newenvironment{thm}{\begin{prethm}{\hspace{-0.5
em}{\bf.}}}{\end{prethm}}
\newtheorem{prepro}{{\bf Theorem}}

\newtheorem{precor}[prethm]{{\bf Corollary}}
\newenvironment{cor}{\begin{precor}{\hspace{-0.5
em}{\bf.}}}{\end{precor}}
\newtheorem{preconj}[prethm]{{\bf Conjecture}}
\newenvironment{conj}{\begin{preconj}{\hspace{-0.5
em}{\bf.}}}{\end{preconj}}
\newtheorem{preremark}[prethm]{{\bf Remark}}
\newenvironment{remark}{\begin{preremark}\em{\hspace{-0.5
em}{\bf.}}}{\end{preremark}}
\newtheorem{prelem}[prethm]{{\bf Lemma}}
\newenvironment{lem}{\begin{prelem}{\hspace{-0.5
em}{\bf.}}}{\end{prelem}}
\newtheorem{preque}[prethm]{{\bf Question}}

\newtheorem{preobserv}[prethm]{{\bf Observation}}

\newtheorem{predef}[prethm]{{\bf Definition}}

\newtheorem{preproposition}[prethm]{{\bf Proposition}}

\newtheorem{preproof}{{\bf Proof.}}
\newtheorem{preprooff}{{\bf Proof}}

\newenvironment{proof}[1]{\begin{preproof}{\rm
#1}\hfill{$\Box$}}{\end{preproof}}

\newtheorem{preproofs}{{\bf The second proof of }}

\newtheorem{preprooft}{{\bf Third proof of }}

\newtheorem{preproofF}{{\bf Proof of}}

\title{\bf\Large 
Edge-decompositions of $O(m)$-edge-connected graphs into isomorphic copies of a fixed tree of size $m$
}
\author{
{\normalsize{\sc Morteza Hasanvand${}$}
}\vspace{3mm}
\\{\footnotesize{${}$\it Department of Mathematical
 Sciences, Sharif
University of Technology, Tehran, Iran}}
{\footnotesize{}}\\{\footnotesize{ $\mathsf{morteza.hasanvand@alum.sharif.edu }$ }}}
\date{}
\begin{document}
\maketitle
\begin{abstract}{
In this paper, we show that every $O(m)$-edge-connected simple graph $G$ of size divisible by $m$ with minimum degree at least $2^{O(m)}$ has an edge-decomposition into isomorphic copies of any given tree $T$ of size $m$. Moreover, the minimum degree condition can be dropped for graphs $G$ with girth greater than the diameter of $T$. These results improve two results due to Bensmail, Harutyunyan, Le, Merker, and Thomass\'e (2017) and Merker (2017) who gave a factorial upper bound on the necessary edge-connectivity.
} {\small
}
\end{abstract}
%
%
%
%
%
%
%
%
%
%
%
%
\section{Introduction}
In this paper, graphs may have multiple edges but loops are avoided and simple graphs have neither multiple edges nor loops.
Let $G$ be a graph. The vertex set, the edge set, and the minimum degree of $G$ are denoted by $V(G)$, $E(G)$, and $\delta(G)$.
The degree $d_G(v)$ of a vertex $v$ is the number of edges of $G$ incident to $v$.
The {\it girth} of $G$ is the size of a shortest cycle that is denoted by ${\it girth(G)}$.
If $G$ has no cycle, the girth is defined to be infinity.
Likewise, we denoted by ${\it diam(T)}$ the {\it diameter} of $G$ that is the maximum distance between any two vertices of $G$.
If $G$ is disconnected, the diameter is defined to be infinity. 
Note that when $G$ is a tree, the diameter is the size of a longest path.
For a partition $P$ of $V(G)$, we denote by $e_G(P)$ the number of edges of $G$ connecting different parts of $P$.
A {\it factor} of a graph refers to a spanning subgraph, and an {\it even factor} is a spanning subgraph with even degrees.
For a graph $G$, we denote by $G[X,Y]$ the bipartite induced factor of $G$ with the bipartition $(X,Y)$.
In addition, we denote by $G[X]$ the induced factor with the vertex set $X$.
The {\it bipartite index $bi(G)$} of a graph $G$ refers to the minimum number of edges for deleting to make the graph bipartite.
For a graph $G$, we say that $G$ is {\it modulo $m$-regular}, if its vertices have degrees divisible by $m$.
For a bipartite graph $G$, we say that $G$ is {\it modulo semi-$m$-regular}, if vertices of one side have degrees divisible by $m$.
A graph $G$ is said to be
{\it $m$-edge-connected}, if for any nonempty proper vertex set $A$,
 $d_G(A) \ge m$, where $d_G(A)$ denotes the number of edges of $G$ with exactly one end in $A$.
Likewise, a graph $G$ is said to be {\it odd-$m$-edge-connected}, if for any vertex set $A$ with $d_G(A)$ odd, $d_G(A) \ge m$.
A graph $G$ is said to be {\it $m$-tree-connected}, if it has $m$ edge-disjoint spanning trees.
By the result of Nash-Williams and Tutte~\cite{Nash-Williams-1961, Tutte-1961}, every $2m$-edge-connected graph is also $m$-tree-connected.
Likewise, a graph $G$ is said to be
{\it $(m,l)$-partition-connected}, if it can be edge-decomposed into an $m$-tree-connected factor and a factor $F$ having an orientation such that for each vertex $v$, $d^+_F(v)\ge l(v)$, where $l$ is a nonnegative integer-valued function on $V(G)$ and $d^+_F(v)$ denotes the out-degree of $v$ in $F$.
Note that an $(m_1+m_2, l_1+l_2)$-partition-connected graph can be edge-decomposed into an $(m_1, l_1)$-partition-connected factor and an $(m_2, l_2)$-partition-connected factor.
For two integers $x$ and $y$, we say that $x\stackrel{k}{\equiv}y$, if $x-y$ is divisible by $k$.
Let $T$ be a tree.
A vertex in $T$ is said to be {\it internal} if it has degree at least two. We denote by $I(T)$ the set of internal vertices of $T$.
A {\it homomorphic} copy $T'$ of $T$ is a graph obtained from $T$ by identifying some of its vertices.
More precisely, there must be a surjective mapping $f : V (T)\rightarrow V(T')$ such that $uv$ is an edge of $T$ if and only if 
$f(u)f(v)$ is an in edge of $T'$.
We say that a graph $G$ admits a {\it $T^*$-edge-decomposition}, if the edges of $G$ can be decomposed into homomorphic copies of $T$. If $f$ is bijective, then $T'$ is also called an isomorphic copy of $T$.
 We say that a graph $G$ admits a {\it $T$-edge-decomposition}, if the edges of $G$ 
can be decomposed into isomorphic copies of $T$.
If $\mathcal{A}$ is a $T^*$-edge-decomposition (or $T$-edge-decomposition), 
for any pair $v\in V(G)$ and $t\in V(T)$, we denote by $\mathcal{A}(v, t)$ the set of all homomorphic copies in $\mathcal{A}$ including $v$ with the preimage $t$.
In addition, we denote by $\mathcal{A}(v)$ the set of all of homomorphic copies in $\mathcal{A}$ including $v$.
Let $(A,B)$ be the bipartition of $T$.
For a bipartite graph $G$ with bipartition $(X,Y)$, we say that $G$ admits 
a {\it $(B,Y)$-compatible} $T$-edge-decomposition (resp. $T^*$-edge-decomposition), 
if every copy of $T$ its vertices in $Y$ are image of vertices in $B$.
Consider an edge-coloring of $T$ with $m$ colors $1,\ldots, m$, where $m=|E(T)|$. 
Assume that $G$ admits a $(B,Y)$-compatible $T^*$-edge-decomposition. Let $F_1,\ldots, F_m$ be a factorization of $G$ such that every factor $F_i$ consists of all edges in $G$ having the color $i$ in a copy of $T$. 
It is not difficult to see that for all $v\in Y$,
$d_{F_i}(v)=d_{F_j}(v)$, where $i$ and $j$ are the colors of two arbitrary edges in $T$ incident with a common vertex in $B$.
Likewise, for all $v\in X$,
$d_{F_i}(v)=d_{F_j}(v)$, where $i$ and $j$ are the colors of two arbitrary edges in $T$ incident with a common vertex in $A$.
We call such a factorization {\it $(B,Y)$-compatible $T$-equitable factorization} or (briefly $T$-equitable factorization).
It was proved in~\cite[Lemma 3.2]{Merker-2017} that a bipartite graph admits a compatible $T$-equitable factorization if and only if it admits a compatible $T^*$-edge-decomposition. Throughout this article, all trees are nontrivial, all variables $\lambda$ are nonnegative integers, and all variables $b$ and $k$ are positive integers.

%

In 2006 Bar{\'a}t and Thomassen proposed the following interesting conjecture on graph edge-decompositions.
\begin{conj}{\rm (\cite{Barat-Thomassen-2006})}\label{intro:conj:Barat-Thomassen}
{For every tree $T$ of size $m$, there exists a positive integer $k_T$ such that every $k_T$-edge-connected simple graph of size divisible by $m$ admits a $T$-edge-decomposition.
}\end{conj}
To prove this conjecture, the first attempts were made by Thomassen (2008)~\cite{Thomassen-2008-P3, Thomassen-2008-P2n} who proved it for the $3$-path and the $4$-path. Later, 
Bar{\'a}t and Gerbner (2012)~\cite{Barat-Gerbner-2014} proved it for a small bistar of size $4$, and 
Thomassen (2012, 2013)~\cite{Thomassen-2012, Thomassen-2013-path, Thomassen:bistars:2013} proved it for every path whose length is a power of two, all stars, and a certain family of bistars. 
After several years, Botler, Mota, Oshiro, and Wakabayashi (2016, 2017)~\cite{Botler-Mota-Oshiro-Wakabayashi-2016, Botler-Mota-Oshiro-Wakabayashi-2017} confirmed the conjecture for 
all paths and Merker (2017)~\cite{Merker-2017} confirmed it for trees with diameter at most~$4$. 

Finally, Bensmail, Harutyunyan, Le, Merker, and Thomass\'e (2017) \cite{Bensmail-Harutyunyan-Le-Merker-Thomasse-2017} proved Conjecture~\ref{intro:conj:Barat-Thomassen} completely by giving the following factorial upper bound 
on $k_T$ for trees $T$ of size $m$. %
The parameter $f(m,\lambda)$ was already defined in \cite{Merker-2017} which is at least $(\lambda +1) m^m $ (by reviewing poofs).
\begin{thm}{\rm (\cite{Bensmail-Harutyunyan-Le-Merker-Thomasse-2017})}\label{intro:thm:proving-conjecture}
{Every $(8m^{2m+3}+8f (m, m10^{50 m}) +12m)$-edge-connected simple graph of size divisible by $m$ 
 admits an edge-decomposition into isomorphic copies of any given fixed tree of size $m$.
}\end{thm}

For stars, Thomassen (2012) proved 
a quartic upper bound on $k_T$, and next Lov{\'a}sz, Thomassen, Wu, and Zhang (2013)~\cite{Lovasz-Thomassen-Wu-Zhang-2013}
refined Thomassen's result by replacing the upper bound by a linear bound.
\begin{thm}{\rm (\cite{Lovasz-Thomassen-Wu-Zhang-2013, Modulo-Orientations-Bounded})}\label{thm:intro:star}
{Every $(3m-3)$-edge-connected simple graph of size divisible by $m$ has an $m$-star-edge-decomposition.
}\end{thm}

For paths, Botler, Mota, Oshiro, and Wakabayashi (2017)~\cite{Botler-Mota-Oshiro-Wakabayashi-2017} presented 
a quartic upper bound on $k_T$ in bipartite graphs. 
Next, Bensmail, Harutyunyan, Le, and Thomass\'e (2018) \cite{Bensmail-Harutyunyan-Le-Thomasse-2019}
showed that dependence on edge-connectivity is not inherently necessary by reducing the upper bound to the constant number $24$ in graphs with sufficiently large minimum degrees. 
Finally, ~Klimo\v{s}ov\'{a} and Thomass\'e (2019)~\cite{Klimova-Thomasse-P3} improved their result to the following sharp version.
\begin{thm}{\rm (\cite{Klimova-Thomasse-P3})}
{For every path $P$, there exists a positive integer $l_P$ such that every
 $3$-edge-connected simple graph of size divisible by $|E(P)|$ with minimum degree at least $l_P$ has a $P$-edge-decomposition.
}\end{thm}
%
%

In this paper, we improve Theorem~\ref{intro:thm:proving-conjecture} by reducing the upper bound on the necessary edge-connectivity to a linear bound and reducing the upper bound on the necessary minimum degree to an exponential bound.
In particular, we show that the minimum degree condition can be dropped for graphs $G$ with girth greater than the diameter of $T$.
\begin{thm}\label{intro:thm:quadratic}
{Every $O(m)$-edge-connected simple graph $G$ of size divisible by $m$ with minimum degree at least $2^{O(m)}$ has an edge-decomposition into isomorphic copies of $T$. In particular, the minimum degree condition can be dropped for graphs $G$ with girth greater than the diameter of $T$
}\end{thm}

 This is also an improvement of a result due to Merker (2017)~\cite[Theorem 6.2]{Merker-2017} who provided an implicit factorial upper bound of $8m^{2m+3}+8f (m, 0) +12m$ on the necessary edge-connectivity for graphs $G$ with girth greater than the diameter of $T$.

%

%
\section{Basic tools: partition-connected graphs}
In this section, we state some preliminary results and provide a simple proof for each of them based on the following theorem. This result can be proved by a combination of Frank's Orientation Theorem~\cite[Theorem 1]{Frank-1980} and Edmonds's Branching Theorem~\cite{Edmonds-1973}.
\begin{thm}{\rm (See \cite{Frank-Optimization-2011, West-Wu-2012})}\label{thm:partition-connected}
{A graph $G$ is $(m,l)$-partition-connected if and only if for every partition $P$ of $V(G)$,
$$e_G(P)\ge m(|P|-1)+\sum_{\{v\}\in P}l(v),$$
where $m$ is a nonnegative integer and $l$ is a nonnegative integer-valued function on $V(G)$. 
}\end{thm} 
Let us start with the following theorem which gives a sufficient condition for the existence of partition-connected factors with small degrees on a specified independent set.
\begin{thm}{\rm (\cite{P})}\label{thm:small-degree}
{Let $\varepsilon $ be a real number with $0<\varepsilon \le 1$. Let $G$ be a graph with an independent set $X\subseteq V(G)$.
If $G$ is $(\lceil m/\varepsilon \rceil ,\lceil l/\varepsilon \rceil )$-partition-connected, then it has an $(m,l)$-partition-connected factor $H$ 
such that for each~$v\in X$, $$d_H(v)\le \lceil \varepsilon d_G(v)\rceil,$$
where $m$ is a nonnegative integer and $l$ is a nonnegative integer-valued function on $V(G)$. 
}\end{thm} 
\begin{proof}
{For each $v\in X$, define $l'(v)=\lfloor (1-\varepsilon) d_G(v)\rfloor$, and for each $v\in V(G)\setminus X$, define $l'(v)=0$.
It is enough to show that $G$ is $(m,l+l')$-partition-connected.
This implies that $G$ can be edge-decomposed into an $(m,l)$-partition-connected factor $H$ and a $(0,l')$-partition-connected factor $F$.
Thus for each vertex $v\in X$, $d_H(v)= d_G(v)-d_F(v)\le d_G(v)-l'(v)=\lceil \varepsilon d_G(v)\rceil$.
Let $P$ be a partition of $V(G)$, let $S$ be the set of vertices $v$ with $\{v\}\in P$.
 Since $G$ is $(\lceil m/\varepsilon \rceil ,\lceil l/\varepsilon \rceil )$-partition-connected, by Theorem~\ref{thm:partition-connected}, we must have 
$\varepsilon e_G(P)\ge m (|P|-1)+\sum_{v\in S} l(v)$.
Since $S\cap X$ is an independent set, we must also have $ (1-\varepsilon)e_G(P)\ge \sum_{v\in S\cap X} (1-\varepsilon) d_G(v)$.
Therefore,
$$ e_G(P)\ge m (|P|-1)+\sum_{v\in S}l(v)+\sum_{v\in S\cap X}(1-\varepsilon) d_G(v) \ge m (|P|-1)+\sum_{\{v\}\in P}( l(v)+l'(v)).$$
Hence $G$ is $(m,l+l')$-partition-connected by Theorem~\ref{thm:partition-connected}, and the proof is completed.
}\end{proof}
The following theorem gives a sufficient edge-connectivity condition for a graph to be partition-connected.%
\begin{thm}{\rm (\cite{ClosedTrails})}\label{thm:edge-connected-partition-connected}
{Every $2m$-edge-connected graph $G$ satisfying $d_G(v) \ge 2m+2l(v)$ for each vertex $v$ 
is also $(m,l)$-partition-connected (even by deleting at most $m$ arbitrary edges), where $m$ is a nonnegative integer and $l$ is a nonnegative integer-valued function on $V(G)$. 
}\end{thm}
\begin{proof}
{Let $M$ be a factor of $G$ with size at most $m$ and let $G'=G\setminus E(M)$. Let $P$ be a partition of $V(G)$. Then
$$e_{G'}(P)=\sum_{A\in P}\frac{1}{2} d_{G'}(A) =
\sum_{A\in P, |A|\ge 2}\frac{1}{2} d_G(A)-e_M(P)+\sum_{\{v\}\in P}\frac{1}{2} d_G(v)
\ge m(|P|-1)+\sum_{\{v\}\in P}l(v).$$
Thus by Theorem~\ref{thm:partition-connected}, the graph $G'$ is $(m,l)$-partition-connected, and the proof is completed.
}\end{proof}
The following theorem gives a sufficient partition-connectivity condition for the existence of a bipartite partition-connected factor.
The edge-connected version of this result was already introduced in \cite[Proposition 1]{Thomassen-2008-P2n}.

\begin{thm}{\rm (\cite{B})}\label{thm:bipartite}
{Every $(2m,2l)$-partition-connected graph $G$ has a bipartite $(m,l)$-partition-connected factor $H$, where $m$ is a nonnegative integer and $l$ is a nonnegative integer-valued function on $V(G)$. In addition, for every vertex set $A$, $d_H(A)\ge d_G(A)/2$.
}\end{thm}
\begin{proof}
{Let $G$ be a $(2m,2l)$-partition-connected graph. 
Let $H$ be a bipartite factor of $G$ with maximum $|E(H)|$. We claim that $d_H(A)\ge d_G(A)/2 $ for every vertex set $A$. Suppose, to the contrary, that $d_H(A)< d_G(A)/2 $ for a vertex set $A$.
Define $X_0=(X\setminus A)\cup (Y \cap A)$ and $Y_0=(Y\setminus A)\cup (X \cap A)$, where $(X,Y)$ is the bipartition of $H$. It is not difficult to see that the graph $G[X_0,Y_0]$ is a bipartite factor of $G$ with more edges than $H$ which is a contradiction.
Let $P$ be a partition of $V(H)$. By Theorem~\ref{thm:partition-connected}, we must have 
$e_G(P)\ge 2m(|P|-1)+\sum_{\{v\}\in P}2l(v)$. Thus
$$e_H(P)=\sum_{A\in P}\frac{1}{2} d_H(A)\ge \sum_{A\in P}\frac{1}{4} d_G(A)=\frac{1}{2}e_G(P) \ge m(|P|-1)+\sum_{\{v\}\in P}l(v).$$
Therefore, by Theorem~\ref{thm:partition-connected}, the graph $H$ is $(m,l)$-partition-connected, and the proof is completed.
}\end{proof}
%
%
%
%
%
%
%

\section{Preliminary results: $T$-equitable factorizations with large minimum degrees}
\label{sec:equitable coloring}
%
%
%
%
%
%
%
%
%
%

The following theorem was originally introduced by Merker~(2017) for finding edge-decompositions of graphs into homomorphic copies of a fixed tree. Shortly thereafter, it was used for proving Theorem~\ref{intro:thm:proving-conjecture}.
In this section, we formulate an explicit linear upper bound on $f(m,0)$ provided that the ratio of $m_b/m_0$ is small enough
and all integers $m_1,\ldots, m_{b}$ are equal to $2$, except possibly $m_b$.
In addition, we strengthen it to guarantee the existence of factors $G_i$ having large minimum degrees whenever $G$ itself has large enough minimum degree. This feature plays an important role in the proof of Theorems~\ref{intro:thm:quadratic}.
\begin{thm}{\rm (\cite[Theorem 3.3]{Merker-2017})}\label{thm:Merker:factorization}
{For any two positive integers $m$ and $\lambda$, there exists a positive integer $f(m,\lambda)$ such that the following holds:
If $m, m_0,\ldots, m_b$ are positive integers satisfying $m= \sum_{0\le i\le b }m_i$, then every $f(m,\lambda)$-edge-connected bipartite simple graph $G$ with bipartition $(X,Y)$ in which each vertex in $X$ has degree divisible by $m$ can be edge-decomposed into $\lambda$-edge-connected factors $G_0, \ldots, G_{b}$ satisfying the following properties:
\begin{enumerate}{
\item [$\bullet$] 
For each $v\in X$, $d_{G_i}(v)=\frac{m_i}{m}d_G(v)$, when $i\in\{0,\ldots, b\}$.
\item [$\bullet$] 
For each $v\in Y$, $d_{G_i}(v)$ is divisible by $m_i$, when $i\in\{1,\ldots, b\}$.
}\end{enumerate}
}\end{thm}
\subsection{Almost even factorizations}
In this subsection, we improve the edge-connectivity needed in Theorem~\ref{thm:Merker:factorization} to $O(\frac{m}{m_0})$ provided that all integers $m_1,\ldots, m_{b}$ are equal to the same number $2$. Our proof is based on the following recent result about the existence of even factors with restricted degrees. 
\begin{thm}{\rm (\cite{EquitableFactorizations-Highly})}\label{thm:G0:edge-decomposition:Eulerian}
{Let $G$ be a graph, let $b$ be a positive integer, let $\varepsilon$ be a real number with $0< \varepsilon \le 1$.
If $G$ is odd-$\lceil \frac{1}{\varepsilon} \rceil$-edge-connected, then
$G$ can be edge-decomposed into $G_0,\ldots, G_b$ factors such that for each vertex $v$, 
\begin{enumerate}{
\item [$\bullet$] 
 $ | d_{G_0}(v)- \varepsilon d_G(v) | <2$. 
\item [$\bullet$] 
 $d_{G_i}(v)$ is even and 
$| d_{G_i}(v)-\frac{1-\varepsilon}{b} d_G(v)|<2$, when $i\in\{1\ldots, b\}$.
}\end{enumerate}
}\end{thm}
Now, we are in a position to significantly refine Theorem~\ref{thm:Merker:factorization} for a particular case.
It should be pointed out that a weaker version of the special cases $m_0\in \{1,2\}$ and $l=0$ of this result were already proved in \cite[Lemma 4.3 and Corollary 4.5]{Botler-Mota-Oshiro-Wakabayashi-2017}.
\begin{cor}\label{cor:parity}
{Let $m, m_0, \ldots, m_{b}$ be positive integers satisfying $m=\sum_{0\le i \le b} m_i$ and $m_1=\cdots = m_{b}=2$. Let $G$ be a bipartite graph with bipartition $(X,Y)$ in which each vertex in $X$ has degree divisible by $m$, and let $l$ be a nonnegative integer-valued function on $V(G)$.
If $G$ is $(\lceil \frac{m}{m_0}\rceil, ml )$-partition-connected, 
then it can be edge-decomposed into factors $G_0, \ldots, G_{b}$ such that
\begin{enumerate}{
\item [$\bullet$] 
For each vertex $v$, $d_{G_i}(v)\ge m_il(v)$, when $i\in\{0\ldots, b\}$.
\item [$\bullet$] 
For each $v\in X$, $d_{G_i}(v)=\frac{m_i}{m}d_G(v)$, when $i\in\{0\ldots, b\}$.
\item [$\bullet$] 
For each $v\in Y$, $d_{G_i}(v)$ is even, when $i\in\{1\ldots, b\}$.
}\end{enumerate}
Moreover by replacing the weaker condition $d_{G_0}(v)\ge m_0l(v)-1$, it is suffices that $G$ is odd-$\lceil \frac{m}{m_0}\rceil$-edge-connected and for each vertex $v$, $d_G(v)\ge ml(v)$.
}\end{cor}
\begin{proof}
{By applying Theorem~\ref{thm:G0:edge-decomposition:Eulerian} with $\varepsilon=m_0/m$, the graph $G$ can be edge-decomposed into factors $G_0, \ldots, G_{b}$ such that for each vertex $v$, 
 $ d_{G_0}(v)\ge 
 \lfloor \frac{m_0}{m} d_G(v) \rfloor -1$
and for all $i\in\{1\ldots, b\}$, $d_{G_i}(v)$ is even and 
$ \lfloor \frac{2}{m}d_G(v)\rfloor -1 \le d_{G_i}(v) \le 
 \lceil \frac{2}{m}d_G(v)\rceil +1 $. Note that $\frac{1-\varepsilon}{b}=\frac{2b/m}{b}=\frac{2}{m}$.
Since $d_G(v)\ge \lceil \frac{m}{m_0}\rceil +ml(v)$, we must have
$ d_{G_0}(v)\ge \lfloor \frac{m_0}{m}(\lceil \frac{m}{m_0}\rceil +ml(v) )\rfloor -1\ge m_0l(v)$.
If $x\in X$, then one can conclude that
$d_{G_i}(x)=\frac{2}{m} d_G(x)\ge 2 l(x)=m_il(v)$, because both integers $d_{G_i}(x)$ and $\frac{2}{m}d_G(x)$ have the same parity. 
This implies that
 $d_{G_0}(x)=\frac{m_0}{m} d_G(x)$.
In addition, if $y\in Y$, then $d_{G_i}(y)\ge
\lfloor \frac{2}{m}d_G(y)\rfloor -1\ge
 2l(y)-1$, because $d_G(y) \ge ml(y)$. 
Since both integers $d_{G_i}(y)$ and $2l(y)$ have the same parity, we must have $d_{G_i}(y)\ge 2l(y)$. 
Hence the assertion holds.
}\end{proof}
%
%
%
%
%
%
%
%
%
%
\subsection{Equitable factorizations with large degrees}
In this subsection, we improve the edge-connectivity needed in Theorem~\ref{thm:Merker:factorization} to $O(\frac{m}{m_0}
\sum_{1\le i \le b, m_i\neq 2}m_i)$.
Before doing so, let us recall the following two lemmas from \cite{Modulo-Factors-Bounded, Complementary}. The first one is an improved version of Lemma 4.1 in~\cite{Merker-2017} and the second one is inspired by Lemma 4.3 in \cite{Merker-2017}. (Note that the proof of Lemma 4.3 in \cite{Merker-2017} needs to be repaired in some parts because degrees of vertices of $G$ in $A$ are not necessarily divisible by $m^2$). 
\begin{lem}{\rm (\cite{Modulo-Factors-Bounded, Merker-2017})}\label{lem:modulo:1/2}
{Let $G$ be a bipartite graph with bipartition $(X,Y)$ in which each vertex in $X$ has even degree.
Let $f$ be an integer-valued function on $Y$ satisfying $\sum_{v\in Y}f(v)\stackrel{k}{\equiv}\frac{1}{2}|E(G)|$.
 If $G$ is $(3k-3)$-edge-connected, then
there exists a factor $H$ of $G$ such that 
\begin{enumerate}{
\item [$\bullet$] For each $v\in X$, $d_{H}(v)=\frac{1}{2}d_G(v)$.
\item [$\bullet$] For each $v\in Y$, $d_H(v)\stackrel{k}{\equiv} f(v)$ and $|d_H(v)- \frac{1}{2}d_G(v)|<k$.
}\end{enumerate}
}\end{lem}
\begin{lem}{\rm (\cite{Complementary})}\label{lem:modulo:1/2:F-F0}
{Let $G$ be a bipartite graph with bipartition $(X,Y)$ and let $f$ be an integer-valued function on $V(G)$ satisfying $\sum_{v\in Y}f(v)\stackrel{k}{\equiv}\sum_{v\in X}f(v)$.
Let $F$, $F_0$, and $T$ be three edge-disjoint factors of $G$ such that for each $v\in X$, 
 $$d_F(v)+\frac{1}{2}d_T(v)\le f(v)\le d_G(v)-(d_{F_0}(v)+\frac{1}{2} d_T(v)).$$ 
 If $T$ is $(3k-3)$-edge-connected, then
there exists a factor $H$ of $G$ including $F$ excluding $F_0$ such that 
\begin{enumerate}{
\item [$\bullet$] For each $v\in X$, $d_{H}(v)=f(v)$.
\item [$\bullet$] For each $v\in Y$, $d_H(v)\stackrel{k}{\equiv} f(v)$.
}\end{enumerate}
}\end{lem}
\begin{proof}
{We repeat the proof in \cite{Complementary} to be self-contained. For convenience, let us define $G_0=G\setminus E(F_0)$. By the assumption, if for a vertex $v\in X$, $d_T(v)$ is odd, then
$d_F(v)+\frac{1}{2}d_T(v)<d_{G_0}(v)-\frac{1}{2} d_T(v)$ and so
 $d_T(v)+d_F(v)\neq d_{G_0}(v)$.
Thus there is a factor $T'$ of $G_0\setminus E(F)$ including $T$ such that for each $v\in X$, $d_{T'}(v)=d_T(v)+1$ when $d_T(v)$ is odd, and $d_{T'}(v)=d_T(v)$ when $d_T(v)$ is even.
By the assumption, for each $v\in X$, we have $d_F(v)\le f(v)-d_{T'}(v)/2 \le d_{G_0}(v)-d_{T'}(v)$. 
Therefore, there exists a factor $F'$ of $G_0\setminus E(T')$ including $F$ such that for each $v\in X$,
$d_{F'}(v)= f(v)-d_{T'}(v)/2$.
 For each $v\in Y$, define $f'(v)=f(v)-d_{F'}(v)$.
According to the assumption, we must have 
$\sum_{v\in Y}f'(v)= \sum_{v\in Y}(f(v)-d_{F'}(v))\stackrel{k}{\equiv}\sum_{v\in X}(f(v)-d_{F'}(v))=\sum_{v\in X}\frac{1}{2}d_{T'}(v)=\frac{1}{2}|E(T')|$. 
Since $T'$ is $(3k-3)$-edge-connected, by Lemma~\ref{lem:modulo:1/2}, there exists a factor $H'$ of $T'$ such that 
for each $v\in X$, $d_{H'}(v)=\frac{1}{2}d_{T'}(v)$, and 
 for each $v\in Y$, $d_{H'}(v)\stackrel{k}{\equiv} f'(v)$ and $|d_{H'}(v)- \frac{1}{2}d_{T'}(v)|<k$.
It is not hard to see that $ H' \cup F'$ is the desired factor that we are looking for.
}\end{proof}
The following theorem develops Corollary 6.10 in \cite{Complementary} to a partition-connected version.
\begin{thm}\label{thm:modulo:b=2}
{Let $m$, $m_0$, and $m_1$ be three positive integers satisfying $m = m_0 +m_1$. Let $G$ be a bipartite graph with bipartition $(X,Y)$ in which each vertex in $X$ has degree divisible by $m$. For $i\in \{0,1\}$, let $\lambda_i$ be a nonnegative integer and let $l_i$ be a nonnegative integer-valued function on $V(G)$.
If $G$ is $(\lfloor \frac{3mm_1}{2\min\{m_0, m_1\}}\rfloor +\lceil \frac{m\lambda_0}{m_0}\rceil+\lceil\frac{m\lambda_1}{m_1}\rceil, 
\lceil \frac{m}{m_0}l_0\rceil+\lceil\frac{m}{m_1}l_1\rceil )$-partition-connected, then
$G$ can be edge-decomposed into two factors $G_0$ and $G_1$ such that 
\begin{enumerate}{
\item [$\bullet$] $G_i$ is $(\lambda_i, l_i)$-partition-connected, when $i\in \{0,1\}$.
\item [$\bullet$] For each $v\in X$, $d_{G_i}(v)=\frac{m_i}{m}d_G(v)$, when $i\in \{0,1\}$.
\item [$\bullet$] For each $v\in Y$, $d_{G_1}(v)$ is divisible by $m_1$,
}\end{enumerate}
}\end{thm}
\begin{proof}
{For each $v\in X$, define $f(v)=\frac{m_1}{m}d_G(v)$, and
for each $v\in Y$, define $f(v)=0$.
Obviously, $\sum_{v\in Y}f(v)\stackrel{m_1}{\equiv}0\stackrel{m_1}{\equiv} \sum_{v\in X}f(v)$.
By the assumption, we can decompose $G$ into three factors $H_0$, $H_1$, and $H$ such that 
$H_i$ is $(\lceil \frac{m}{m_i}\lambda_i\rceil , \lceil \frac{m}{m_i}l_i\rceil )$-partition-connected
and
$H$ is $\lceil \frac{m(3m_1-1)}{2\min\{m_0, m_1\}} \rceil $-tree-connected.
According to Theorem~\ref{thm:small-degree}, 
 the graph $H_i$ has a $(\lambda_i, l_i)$-partition-connected factor $F_i$
 such that for each $v\in X$, $d_{F_i}(v)\le \lceil \frac{m_i}{m} d_{H_i}(v)\rceil$.
In addition, the graph $H$ has a $(3m_1-1)$-tree-connected factor $T'$
 such that for each $v\in X$, $d_{T'}(v)\le \lceil \frac{2 \min \{m_{0}, m_1\}}{m} d_{H}(v)\rceil$.
Let $T$ be a $(3m_1-3)$-tree-connected factor of $T'$ such that for each vertex $v$, $d_T(v)\le d_{T'}(v)-2$.
Therefore, for each $v\in X$, 
$d_{F_1}(v) +\frac{1}{2}d_{T}(v)\le 
\lceil \frac{ m_{1}}{m} d_{H_1}(v)\rceil +\lfloor \frac{m_1}{m} d_{H}(v)\rfloor
\le \lceil\frac{m_1}{m} d_G(v)\rceil= f(v)$.
Similarly, for each $v\in X$, 
$d_{F_0}(v) +\frac{1}{2}d_{T}(v)\le 
\lceil \frac{ m_{0}}{m} d_{H_0}(v)\rceil +\lfloor \frac{m_0}{m} d_{H}(v)\rfloor \le \lceil\frac{m_0}{m} d_G(v)\rceil=d_G(v)-f(v)$.
By Lemma~\ref{lem:modulo:1/2:F-F0}, the graph $G$ has a factor $G_1$ including $F_1$ excluding $F_0$ such that 
for each $v\in X$, $d_{G_1}(v)=f(v)=\frac{m_1}{m}d_G(v)$, and for each $v\in Y$, $d_{G_1}(v)\stackrel{m_1}{\equiv} f(v)=0$.
It is enough to set $G_0=G\setminus E(G_1)$ to complete the proof.
}\end{proof}
Now, we are in a position to significantly refine Theorem~\ref{thm:Merker:factorization} for another particular case.
In fact, this version allows to us to select $m_b$ in Corollary~\ref{cor:parity} arbitrarily but by increasing the required partition-connectivity.

\begin{cor}\label{cor:parity:factorization}
{Let $m, m_0, m_1,\ldots, m_{b}$ be positive integers satisfying $m=\sum_{0\le i \le b} m_i$ and $m_1=\cdots = m_{b-1}=2$. Let $G$ be a bipartite graph with bipartition $(X,Y)$ in which each vertex in $X$ has degree divisible by $m$, and
let $l$ be a nonnegative integer-valued function on $V(G)$.
If $G$ is 
 $(\lceil \frac{3mm_b}{2\min \{m-m_b,m_b\}}+ \frac{2m}{m_0}\rceil, 2ml)$-partition-connected, then
$G$ can be edge-decomposed into factors $G_0,\ldots, G_{b}$ such that 
\begin{enumerate}{
\item [$\bullet$] For each vertex $v$, $d_{G_i}(v)\ge m_il(v)$, when $i\in \{0,\ldots, b\}$.
\item [$\bullet$] For each $v\in X$, $d_{G_{i}}(v)=\frac{m_i}{m}d_G(v)$, when $i\in \{0,\ldots, b\}$.
\item [$\bullet$] For each $v\in Y$, $d_{G_{i}}(v)$ is divisible by $m_i$,
when $i\in \{1,\ldots, b\}$.
}\end{enumerate}
In addition, by ignoring the condition $d_{G_b}(v)\ge m_bl(v)$,
 it suffices that $G$ is $(\lceil \frac{3mm_b}{2\min \{m-m_b,m_b\}}+ \frac{2m}{m_0}\rceil, ml)$-partition-connected.
}\end{cor}
\begin{proof}
{By the assumption, the graph $G$ is $(\lfloor \frac{3mm_b}{2\min \{m'_0,m_b\}}\rfloor+ 
\lceil \frac{m}{m'_0} \lambda'\rceil, 
 \frac{m}{m'_0}l'+\frac{m}{m_b}l_b )$-partition-connected, 
where $m'_0=m-m_b$, $l'=m'_0 l$, and $l_b=m_bl$, $\lambda'=\lceil \frac{m'_0}{m_0}\rceil $. 
Note that $\lceil \frac{m}{m'_0} \lambda'\rceil\le \frac{m}{m_0}+\frac{m}{m'_0}\le \frac{2m}{m_0}$.
According to Theorem~\ref{thm:modulo:b=2}, the graph $G$ contains a $(0,l_b)$-partition-connected factor $G_b$
 such that its complement $G'$ is
$(\lambda',l')$-partition-connected,
for each $v\in X$, $d_{G'}(v)=\frac{m'_0}{m}d_G(v)$ and $d_{G_{b}}(v)=\frac{m_b}{m}d_G(v)$, and
for each $v\in Y$, $d_{G_{b}}(v)$ is divisible by $m_b$.
By applying Corollary~\ref{cor:parity} with setting $m'_0$ instead of $m$, the graph $G'$ can be edge-decomposed into factors $G_0,\ldots, G_{b-1}$ such that
 for each vertex $v$, $d_{G_i}(v)\ge m_i l(v)$,
for each $v\in X$, $d_{G_{i}}(v)=\frac{m_i}{m'_0}d_{G'}(v)=\frac{m_i}{m}d_G(v)$, when $i\in \{0,\ldots, b-1\}$,
and for each $v\in Y$, $d_{G_i}(v)$ is even when $i\in\{1,\ldots, b-1\}$. For proving the last statement, it is enough to set $l_b=0$.
Hence the assertion holds.
}\end{proof}
\begin{lem}{\rm (\cite{Werra-1971})}\label{lem:Werra}
{Every bipartite graph $G$ can be edge-decomposed into $m$ factors $F_1,\ldots, F_m$
 such that for each $v\in V(F_i)$ with $1\le i\le m$, 
 $$|d_{F_i}(v)-d_{G}(v)/m|<1.$$
Consequently, if $G$ is a modulo $m$-regular bipartite graph, then it can be decomposed into factors $G_1,\ldots, G_b$ such that for each $v\in V(G_i)$, $d_{G_i}(v)=\frac{m_i}{m}d_{G}(v)$, where $m_1,\ldots, m_{b}$ are positive integers satisfying $m=\sum_{1\le i \le b} m_i$.
}\end{lem}
\begin{proof}
{To prove the consequence, it is enough to set $G_i$ to be
the union of $m_i$ edge-disjoint factors $F_j$ so that all of $G_i$ are edge-disjoint.
Note that for each vertex $v$, $d_{F_j}(v)=\frac{1}{m}d_{G}(v)$ and hence $d_{G_i}(v)=\frac{m_i}{m}d_{G}(v)$.
}\end{proof}
The following corollary enables us to select $m_i$ in Corollary~\ref{cor:parity} arbitrarily but by increasing the required partition-connectivity.
\begin{cor}
{Let $m, m_0, m_1,\ldots, m_{b}$ be positive integers satisfying $m=\sum_{0\le i \le b} m_i$ and let 
$m'_b=\sum_{1\le i\le b, m_i\neq 2}m_i\ge 2$. 
Let $G$ be a bipartite graph with bipartition $(X,Y)$ in which each vertex in $X$ has degree divisible by $m$, and
let $l$ be a nonnegative integer-valued function on $V(G)$.
If $G$ is 
 $(\lceil \frac{3mm'_b}{2\min \{m-m'_b,m'_b\}}+ \frac{2m}{m_0}\rceil, 2ml)$-partition-connected, then
$G$ can be edge-decomposed into factors $G_0,\ldots, G_{b}$ such that 
\begin{enumerate}{
\item [$\bullet$] For each vertex $v$, $d_{G_i}(v)\ge m_il(v)$, when $i\in \{0,\ldots, b\}$.
\item [$\bullet$] For each $v\in X$, $d_{G_{i}}(v)=\frac{m_i}{m}d_G(v)$, when $i\in \{0,\ldots, b\}$.
\item [$\bullet$] For each $v\in Y$, $d_{G_{i}}(v)$ is divisible by $m_i$,
when $i\in \{1,\ldots, b\}$.
}\end{enumerate}
}\end{cor}
\begin{proof}
{We may assume that $m_i=2$ for all $i$ with $1\le i\le p$, and $m'_b=\sum_{p< i\le k}m_i$.
By Corollary~\ref{cor:parity:factorization}, the graph $G$ can be edge-decomposed into $p+2$ factors $G_0,\ldots, G_{p}$ and $H$
such that for each vertex $v$, $d_{G_i}(v)\ge m_il(v)$ when $i\in \{0,\ldots, p\}$,
for each $v\in X$, $d_{G_{i}}(v)=\frac{m_i}{m}d_G(v)$ when $i\in \{0,\ldots, p\}$, and
for each $v\in Y$, $d_{G_{i}}(v)$ is even when $i\in \{1,\ldots, p\}$.
In addition, for each vertex $v$, $d_{H}(v)\ge m'_{b}l(v)$,
for each $v\in X$, $d_{H}(v)=\frac{m'_{b}}{m}d_G(v)$, and
for each $v\in Y$, $d_{H}(v)$ is divisible by $m'_{b}$.
By Lemma~\ref{lem:Werra}, the graph $H$ can be edge-decomposed into factors $G_{p+1},\ldots, G_{b}$
such that for each vertex $v$, $d_{G_i}(v)=\frac{m_i}{m'_{b}}d_H(v)$, where $i\in \{p+1,\ldots, b\}$. Thus 
for each vertex $v$, $d_{G_i}(v) \ge m_il(v)$, 
for each $v\in X$, $d_{G_{i}}(v)=\frac{m_i}{m}d_G(v)$, and
for each $v\in Y$, $d_{G_{i}}(v)$ is divisible by $m_{i}$. Hence the proof is completed.
}\end{proof}
\begin{remark}
{Note that we can impose the quadratic bound of $m^2(3 + 2\lambda)$ on $f(m, \lambda)$ in Theorem~\ref{thm:Merker:factorization}.
To keep simplicity of this paper, we decided to delete its proof here.
}\end{remark}
%
%
%
%
%
%
%
%
%
\section{Bipartite graphs: $T$-edge-decompositions}
\label{sec:diameter}
%
%
\subsection{Preliminary results: compatible $T$-edge-decomposition}
The following assertion gives sufficient conditions for the existence of compatible $T$-edge-decompositions in modulo semi-regular bipartite simple graphs which was implicitly established in \cite{Merker-2017, Bensmail-Harutyunyan-Le-Merker-Thomasse-2017}.
We will employ this result together with Corollary~\ref{cor:parity:factorization} to prove 
Theorem~\ref{thm:base:bipartite:decomposition}.
\begin{thm}{\rm (\cite{Bensmail-Harutyunyan-Le-Merker-Thomasse-2017, Merker-2017})}\label{preliminary:thm:equitable:decomposition}
{Let $T$ be a tree of size $m$ and let $m_0, \ldots, m_b$ be positive integers satisfying $m=\sum_{0\le i\le b}m_i$.
Assume that there is a partition $B_0,\ldots, B_b$ of one partite set $B$ of $T$ such that $B_0$ consists of all leaf vertices in $B$ and
 $m_i=\sum_{v\in B_i}d_T(v)$.
Let $G$ be a bipartite graph with bipartition $(X,Y)$ in which each vertex in $X$ has degree divisible by $m$.
Then $G$ admits a $(B, Y)$-compatible $T^*$-edge-decomposition $\mathcal{A}^*$, if it can be edge-decomposed into factors $G_0, \ldots, G_{b}$ satisfying the following properties:
\begin{enumerate}{
\item [$\bullet$] 
For each $v\in X$, $d_{G_i}(v)=\frac{m_i}{m}d_G(v)$, when $i\in\{0,\ldots, b\}$.
\item [$\bullet$] 
For each $v\in Y$, $d_{G_i}(v)$ is divisible by $m_i$, when $i\in\{1,\ldots, b\}$.
}\end{enumerate}
In addition, for any pair $v\in X$ and $t\in A$, $|\mathcal{A}^*(v,t)|= d_{G}(v)/m$,
and for any pair $v\in Y$ and $t\in B_i$, $|\mathcal{A}^*(v,t)|- d_{G_i}(v)/m_i | < 1$.
In particular, $G$ admits such a $(B, Y)$-compatible $T$-edge-decomposition, if $G$ is simple and one of the following conditions holds:
\begin{enumerate}{

\item [$\bullet$] 
$girth(G)>diam(T)$.

\item [$\bullet$] 
$\delta(G_i)\ge m_i 10^{50m}$, when $i\in\{0,\ldots, b\}$.
}\end{enumerate}
}\end{thm}
\begin{proof}
{By Lemma~\ref{lem:Werra}, the graph $G_i$ has a factorization $F_{i, 1},\ldots, F_{i, m_i}$ such that for any $v\in X$ and $j\in \{1,\ldots, m_i\}$, $d_{F_{i, j}}(v)=\frac{1}{m_i}d_{G_i}(v)=\frac{1}{m}d_G(v)$, 
and for any $v\in Y$ and 
$j\in \{1,\ldots, m_i\}$, $\lfloor \frac{1}{m_i}d_{G_i}(v) \rfloor \le d_{F_{i, j}}(v) \le\lceil \frac{1}{m_i}d_{G_i}(v) \rceil$.
Consequently, $d_{F_{i, j}}(v)=\frac{1}{m_i}d_{G_i}(v)$ when $i\neq 0$.
Therefore, all factors $F_{i, j}$ form a $(B, Y)$-equitable factorization $F_1,\ldots, F_m$ for $G$ (corresponding to an edge-coloring of $T$ with colors $\{1,\ldots, m\}$).
Thus by Lemma 3.2 in \cite{Merker-2017}, the graph $G$ admits a $(B, Y)$-compatible $T^*$-edge-decomposition satisfying the theorem.
Note that if $girth(G)>diam(T)$, then this $T^*$-edge-decomposition must automatically be a $(B, Y)$-compatible $T$-edge-decomposition. 
In addition, if the minimum degree of every $G_i$ is at least $m_i10^{50m}$, then the minimum degree of every $F_j$ must be at least $10^{50m}$. Thus by Theorem~7 in \cite{Bensmail-Harutyunyan-Le-Merker-Thomasse-2017}, the graph $G$ admits a $T$-edge-decomposition satisfying the theorem.
}\end{proof}
%

%
%
%
%
%
\subsection{Bipartite graphs with degree divisible by $m$ in one side}
The following theorem shows an application of Corollary~\ref{cor:parity:factorization} and Theorem~\ref{preliminary:thm:equitable:decomposition}.
\begin{thm}\label{thm:bipartite:semiregular}
{Let $T$ be a tree of size $m$ with bipartition $(A,B)$.
Let $G$ be a bipartite simple graph with bipartition $(X,Y)$ in which vertices in $X$ have degrees divisible by $m$. 
Then $G$ admits a $(B,Y)$-compatible $T$-edge-decomposition, if at least one of the following properties holds:
\begin{enumerate}{
\item [$\bullet$] 
 $ girth(G) >diam(T)$ and $G$ is $\lfloor \frac{9}{2}m\rfloor $-tree-connected.

\item [$\bullet$] 
$diam(T)= 3$ and $G$ is $\lceil \frac{3}{2}m\rceil $-tree-connected.

\item [$\bullet$] 
$G$ is $(\lfloor \frac{9}{2}m\rfloor, 2m10^{50m})$-partition-connected.
}\end{enumerate}
Moreover, if $(\lfloor \frac{9}{2}m\rfloor, 4ml)$-partition-connected, then $G$ admits a $(B, Y)$-compatible $T^*$-edge-decomposition such that for any compatible pair $v\in V(G)$ and $t\in V(T)$, there are $l(v)$ homomorphic copies of $T$ including $v$ with the preimage $t$, where $l$ is a nonnegative integer-valued function on $V(G)$.

}\end{thm}
\begin{proof}
{Assume that $diam(T) =3$ and let $m_1$ the degree of an internal vertex $v$ of $T$ and let $m_0=m-m_1$.
We select $v$ such that $m_1\le (m+1)/2$. This yields that $m_1=\min\{m_0,m_1\}$ or $m_0=m_1-1=(m-1)/2$.
Therefore, the graph $G$ must be $(\lceil \frac{3m(m_1 -1)}{2\min\{m_0,m_1\}}\rceil )$-tree-connected. 
By Theorem~\ref{thm:small-degree}, the graph $G$ has a $(3m_1-3)$-tree-connected factor $H$ 
such that for each $v\in X$, $d_{H}(v)\le 
\frac{2\min\{m_0,m_1\}}{m}d_{G}(v) $. 
Note that $0 < \frac{2\min\{m_0,m_1\}}{m} < 1$.
By Lemma~\ref{lem:modulo:1/2:F-F0}, the graph $G$ has an edge-decomposition into two factors 
$G_{0}$ and $G_{1}$ such that for each $v\in X$, $d_{G_{i}}(v)=\frac{m_i}{m}d_G(v)$ and for each $v\in 
Y$, $d_{G_{1}}(v)$ is divisible by $m_1$. 
Since $girth(G)\ge 4 > diam(T)$, by Theorem~\ref{preliminary:thm:equitable:decomposition}, the graph $G$ admits a $(B,Y)$-compatible $T$-edge-decomposition.

Now, assume that $G$ is $(\lfloor \frac{9}{2}m\rfloor, 2ml)$-partition-connected.
We may assume $A \cap I(T)$ and $B \cap I(T)$ are not empty.
Otherwise, $T$ must be a star and so the assertion follows from Theorem~\ref{thm:intro:star}. 
Let $B_0,\ldots, B_b$ be the partition of $B$ such that $B_0$ consists all leaf vertices of $T$ in $B$, every $B_i$ with $0 < i < b$ consists of one vertex of $T$ in $B$ with degree $2$, and $B_b$ either consists of all vertices of $T$ in $B$ with degree at least $3$ or consists of one vertex of $T$ in $B$ with degree $2$. Let $m_i=\sum_{v\in B_i}d_T(v)$. Notice that $m_0$ denotes be the number of leaf vertices of $T$ in $B$,
and $m_1=\cdots =m_{b-1}=2$, and $b-1$ or $b$ denotes the number of vertices of $T$ with degree $2$ in $B$.
Note also that $m-m_b\ge m_0$.
Since $T$ contains $\max_{v\in A \cap I(T)}(d_T(v)-2)+\max_{v\in B \cap I(T)}(d_T(v)-2)+2$ leaf vertices, 
there are at least $\max_{v\in A \cap I(T)}(d_T(v)-2)+1$ leaf vertices in $A$ or at least
$\max_{v\in B \cap I(T)}(d_T(v)-2)+1$ leaf vertices in $B$.
We may therefore assume that $m_0 \ge \sum_{v\in B \cap I(T)}(d_T(v)-2)+1$.
If $m_b \neq 2$, then 
$$3(m_0-1)\ge \sum_{v\in B \cap I(T)}3(d_T(v)-2) = \sum_{v\in B_b}3(d_T(v)-2) \ge \sum_{v\in B_b}d_T(v) = m_b,$$
which implies that (regardless of $m_b=2$ or not)
$$ \lfloor \frac{9}{2}m \rfloor \ge 
\max\{ \frac{3m}{2}, \frac{3m(3m_0-3)}{2m_0}\}+\frac{2m}{m_0} \ge \frac{3mm_b}{2\min \{m_b, m-m_b\}}+ \frac{2m}{m_0}.$$
Thus by Corollary~\ref{cor:parity:factorization}, we can decompose $G$ into factors $G_{0},\ldots, G_{b}$
such that for each vertex $v$, $d_{G_{i}}(v)\ge m_il(v)$ when $i\in \{0,\ldots, b\}$, for each $v\in X$, $d_{G_{i}}(v)=\frac{m_i}{m}d_G(v)$ when $i\in \{0,\ldots, b\}$, and 
for each $v\in Y$, $d_{G_{i}}(v)$ is divisible by $m_i$ when $i\in \{1,\ldots, b\}$.
Therefore, by Theorem~\ref{preliminary:thm:equitable:decomposition}, the graph $G$ admits a $(B,Y)$-compatible $T^*$-edge-decomposition such that for any pair $v\in X$ and $t\in A$ and any pair $v\in Y$ and $t\in B_i$, there are $l(v)$ homomorphic copies of $T$ including $v$ with the preimage $t$. 
Hence $G$ admits a $(B,Y)$-compatible $T^*$-edge-decomposition with the desired properties.
If $girth(G) > diam(T)$ (regardless of $l=0$ nor not), then 
every homomorphic copy of $T$ must be an isomorphic copy and so $G$ admits a $(B,Y)$-compatible $T$-edge-decomposition automatically.
If $l=10^{50m}$ and $G$ is a simple graph, then according to Theorem~\ref{preliminary:thm:equitable:decomposition}, $G$ admits a $(B,Y)$-compatible $T$-edge-decomposition.
Hence the theorem holds.
}\end{proof}
The following theorem reduces the edge-connectivity needed in Theorem~\ref{thm:bipartite:semiregular} for a special type of trees.
\begin{thm}\label{thm:subdivision-tree}
{Let $T$ be a tree of size $m$ with bipartition $(A,B)$. Assume that $B$ consists of $m_0$ leaves and some vertices with degree $2$, where $0 < m_0 < m$. Let $G$ be a $2$-edge-connected bipartite simple graph in which vertices in one side have degrees divisible by $m$. Then $G$ admits a $(B,Y)$-compatible $T$-edge-decomposition, if at least one of the following properties holds:
\begin{enumerate}{
\item [$\bullet$] 
 $G$ is odd-$\lceil \frac{m}{m_0}\rceil$-edge-connected and $girth(G) > diam(T)$.

\item [$\bullet$] 
$G$ is odd-$\lceil \frac{m}{m_0}\rceil$-edge-connected and $\delta(G) \ge m10^{51m}$.
}\end{enumerate}
}\end{thm}
\begin{proof}
{Let $(X,Y)$ be a bipartition of $G$ in which vertices in $X$ have degrees divisible by $m$. 
For convenience, let us define $m_1= \cdots= m_{b}=2$ where $b=(m-m_0)/2$.
Since $G$ is odd-$\lceil \frac{m}{m_0}\rceil$-edge-connected and $\delta(G) \ge m10^{51m}$, 
by applying Corollary~\ref{cor:parity}, the graph $G$ 
can be edge-decomposed into factors $G_0, \ldots, G_b$ such that
for each vertex $v$, $d_{G_i}(v)\ge m_i10^{51m}-1\ge m_i10^{50m}$ when $i\in\{0\ldots, b\}$,
for each $v\in X$, $d_{G_i}(v)=\frac{m_i}{m}d_G(v)$ when $i\in\{0\ldots, b\}$, 
and for each $v\in Y$, $d_{G_i}(v)$ is even when $i\in\{1\ldots, b\}$.
Thus the assertion follows from Theorem~\ref{preliminary:thm:equitable:decomposition}.
The proof of the first assertion is also similar.
}\end{proof}
%
%
%
%
%
\subsection{Bipartite graphs with large girth or minimum degree}

The following lemma is an important tool for decomposing graphs into two highly partition-connected modulo semi-regular factors $G_1$ and $G_2$ which was originally introduced in \cite[Proposition 2]{Thomassen:bistars:2013} with a weaker version.
Note that two similar versions are also established in \cite[Lemma 2.5]{Botler-Mota-Oshiro-Wakabayashi-2016}, and \cite[Proposition 6]{Merker-2017}.
\begin{lem}\label{lem:G1G2:semi-regular}
{Let $G$ be a bipartite graph with bipartition $(V_1, V_2)$.
Let $\lambda$ be a nonnegative integer, let $l$ be a nonnegative integer-valued function on $V(G)$.
If $G$ is $(2m-2+2\lambda, 2l)$-partition-connected and $|E(G)|$ is divisible by $m$, then $G$ admits an edge-decomposition into two $(\lambda,l)$-partition-connected factors $G_1$ and $G_2$ such that for each $v\in V_i$, $d_{G_i}(v)$ is divisible by $m$.
}\end{lem}
\begin{proof}
{Decompose $G$ into three factors $H_1$, $H_2$, and $H$ such that $H_i$ is $(\lambda_i,l_i)$-partition-connected and $H$ is $(2m-2)$-tree-connected. For each $v\in V_i$, define $p(v)=-d_{H_i}(v)$ (modulo $m$). Since $|E(G)|$ is divisible by $m$, we must have
$|E(H)|=|E(G)|-|E(H_1)|-|E(H_2)|
\stackrel{m}{\equiv}\sum_{v\in V_1} p(v)+\sum_{v\in V_2} p(v)=\sum_{v\in V(H)} p(v)
$. Thus by Corollary 4.5 in \cite{Modulo-Orientations-Bounded}, the graph $H$ has an orientation 
such that for each vertex $v$, $d_H^+(v) \stackrel{m}{\equiv} p(v)$. 
Let $F_i$ be the factor of $H$ consisting of all edges directed away from $V_i$. It is easy to check that $G_1=H_1\cup F_1$ and $G_2=H_2\cup F_2$ are the desired factors that we are looking for.
}\end{proof}
The following theorem develops Theorem~\ref{thm:bipartite:semiregular} to arbitrary bipartite graphs but requires to increase the necessary partition-connectivity by a constant factor.
\begin{thm}\label{thm:base:bipartite:decomposition}
{Let $T$ be a tree of size $m$.
A bipartite simple graph $G$ admits a $T$-edge-decomposition, if at least one of the following conditions holds:
\begin{enumerate}{
\item [$\bullet$] 
 $ girth(G) >diam(T)$ and $G$ is $11m$-tree-connected.
\item [$\bullet$] 
$diam(T)= 3$ and $G$ is $5m$-tree-connected.

\item [$\bullet$] 
$G$ is $(11m, 4m10^{50m})$-partition-connected.
}\end{enumerate}
Moreover, every $(11m, 4ml)$-partition-connected bipartite graph $G$ admits a $T^*$-edge-decomposition such that for any pair $v\in V(G)$ and $t\in V(T)$, there are $l(v)$ homomorphic copies of $T$ including $v$ with the preimage $t$, where $l$ is a nonnegative integer-valued function on $V(G)$.
}\end{thm}
\begin{proof}
{Let $(V_1,V_2)$ be the bipartition of $V(G)$. 
Assume that $diam(T) =3$.
By Lemma~\ref{lem:G1G2:semi-regular}, we can decompose $G$ into two 
$\lceil \frac{3}{2}m\rceil$-tree-connected factors $G_1$ and $G_2$
such that for each $v\in V_i$, $d_{G_i}(v)$ is divisible by $m$.
Therefore, by Theorem~\ref{thm:bipartite:semiregular}, the graph $G_i$ admits a $T$-edge-decomposition and
so does $G$. 
Now, assume that $G$ is $(11m, 4ml)$-partition-connected.
By Lemma~\ref{lem:G1G2:semi-regular}, we can decompose the graph $G$ into two factors $G_1$ and $G_2$
such that $G_i$ is $(\lfloor \frac{9}{2}m \rfloor, 2ml)$-partition-connected for each $v\in V_i$, $d_{G_i}(v)$ is divisible by $m$.
Therefore, by Theorem~\ref{thm:bipartite:semiregular}, the graph $G_i$ admits a $(B,V_i)$-compatible $T^*$-edge-decomposition such that for any pair $v\in V_i$ and $t\in A$ and any pair $v\in V(G)\setminus V_i$ and $t\in B$,
 there are $l(v)$ homomorphic copies of $T$ including $v$ with the preimage $t$. 
Hence $G$ admits a $T^*$-edge-decomposition with the desired properties.
If $girth(G) > diam(T)$ (regardless of $l=0$ nor not), then 
every homomorphic copy of $T$ must be an isomorphic copy and so $G$ admits a $T$-edge-decomposition automatically.
If $l=10^{50m}$ and $G$ is a simple graph, then according to Theorem~\ref{preliminary:thm:equitable:decomposition}, the graph $G_i$ admits a $T$-edge-decomposition and
so does $G$.
Hence the theorem holds.
}\end{proof}
%
%
%
%
%
%
%
\section{Graphs with sufficiently large girth}
In this section, we are going to develop the first item of Theorem~\ref{thm:base:bipartite:decomposition} to arbitrary graphs.
Before stating this result, let us establish the following lemma for working with graphs with small bipartite index.
\begin{lem}\label{lem:edge-extension}
{Let $G$ be a simple graph with a factor $M$, and let $T$ be a tree of size $m$.
If $girth(G)> diam(T)$ and $\delta(G)\ge |E(M)|+2m$, then $M$ can be extended to a subgraph of $G$ having a $T$-edge-decomposition.
In addition, the assertion holds for all trees when $\delta(G)\ge 2|E(M)|+2m^2$.
}\end{lem} 
\begin{proof}
{We may assume there is no copy of $T$ in $M$. Otherwise, we delete it from $M$ and use induction.
Let $t_0$ be a fixed internal vertex of $T$.
We inductively select a copy $M_i$ of a subtree of $T$ including $t_0$ in the remaining factor of $M$ with the maximum size and delete it from $M$.
Following \cite[Page 5]{Barat-Gerbner-2014}, we say that a vertex in $M_i$ is {\it unsaturated} 
if its degree in $M_i$ is less than the degree of its preimage in $T$.
Let $M_j$ be the first subtree including $v$ as an unsaturated vertex.
If $M_{j'}$ is another subtree including $v$ as an unsaturated vertex and $vv_{j'}\in E(M_{j'})$,
then according to the maximality of $M_j$, we must have $v_{j'}\in V(M_j)$. 
Thus there are at most $|V(M_j)|-1$ subtrees including $v$ as an unsaturated vertex.
In addition, there is at most one subtree including $v$ as an unsaturated vertex when $girth(G)> diam(T)$. 
Let $r=1$ when $girth(G)> diam(T)$ and let $r=m-1$ otherwise.
In other words, in both cases, are at most $r$ subtrees including $v$ as an unsaturated vertex.

Let $\mathcal{A}_0$ be a collection of some edge-disjoint copies $T_1,\ldots, T_n$ of subtrees of $T$ such that every $T_i$ contains $M_i$, 
and for every vertex $v$ and every $t\in V(T)$, there are at most $r$ subtrees in $\mathcal{A}_0$ including $v$ with the preimage $t$. 
According to the construction, the collection of all $M_i$ is a candidate for $\mathcal{A}_0$.
We consider $\mathcal{A}_0$ with the maximum size of $T_1\cup \cdots \cup T_n$.
We claim that $\mathcal{A}_0$ is the desired $T$-edge-decomposition. 
Otherwise, there is a subtree $T_i\in \mathcal{A}_0$ having an unsaturated vertex $x$ so that its degree in $T_i$ is less than the degree of its preimage $t_x$ in $T$. 
Let $t_x'$ be the vertex in $V(T)$ incident with $t_x$ such that $t_x$ lies in the path in $T$ connecting $t_0$ to $t'_x$. 
Let $u_1,\ldots, u_p$ be the neighbours of $x$ in $G$ such that every edge $xu_i$ is not contained in subtrees of $\mathcal{A}_0$.
By the assumption on $\mathcal{A}_0$, for every vertex $v$ and every $t\in V(T)$, there are at most $r$ subtrees in $\mathcal{A}_0$ including $v$ with the preimage $t$. 
Thus there are at least $d_G(x)-d_M(x)-2m r+1$ edges incident with $x$ in $G$ which are different from the edges of subtrees in $\mathcal{A}_0$ (whether in $M$ or not). 
This means that $p\ge d_G(x)-d_M(x)-2m r+1$.
If there is at most $r-1$ subtrees in $\mathcal{A}_0$ including $u_j$ with the preimage $t_x'$, then we can add the edge $xu_j$ to $T_i$ to obtain a new collection $\mathcal{A}_0$ with larger size satisfying the desired properties, which is a contradiction. 
Thus for every vertex $u_j$, there are at least $r$ subtrees $T_{1,j},\ldots, T_{r,j}$ in $\mathcal{A}_0$ including $u_j$ with the preimage $t_x'$. Obviously, all subtrees $T_{t,j}$ must be distinct. Thus we have at least $pr$ edges of $M$.
Therefore, $|E(M)|\ge pr\ge r(\delta(G)-d_M(x)-2m r+1)$.
If $r=m-1$, then $2|E(M)|+2m^2 > \delta(G)$, which is a contradiction.
 If $girth(G)> diam(T)$, then every subtree $T_{i, j}$ does not contain edges of $M$ incident with $x$. In this case, 
we have at least $pr+d_M(x)$ edges of $M$ and so $|E(M)|\ge pr+d_M(x)\ge \delta(G)-2m +1$, which is again a contradiction.
}\end{proof}
We also need the following lemma for working with graphs with large bipartite index.
\begin{lem}{\rm (\cite{EquitableFactorizations-Highly})}\label{lem:bi:divisible-by-m}
{If $G$ is a $(2m-2+3\lambda)$-tree-connected graph of size divisible by $m$ and $bi(G)\ge \lambda$, then every vertex $v$ can be split into two vertices $v_1$ and $v_2$ such that resulting graph $H$ is a bipartite $\lambda$-tree-connected graph with bipartition $(V_1,V_2)$ for which $V_i=\{v_i:v\in V(G)\}$
and for each $v\in V(G)$, $d_{H}(v_1)$ is divisible by $m$.
}\end{lem}
Now, we are in a position to prove the second part of Theorem~\ref{intro:thm:quadratic} by giving an explicit upper bound on the necessary edge-connectivity.
\begin{thm}\label{thm:homomorphic:m2}
{Let $T$ be a tree of size $m$.
Let $G$ be a $31m$-edge-connected simple graph with $\delta(G)\ge 35m$.
If $|E(G)|$ is divisible by $m$ and $girth(G) > diam(T)$, then 
then $G$ admits a $T$-edge-decomposition.
}\end{thm}
\begin{proof}
{For notational simplicity, let us write $\lambda=\lfloor \frac{9}{2}m\rfloor $.
By Theorem~\ref{thm:edge-connected-partition-connected}, the graph $G$ is 
$(2m-2+3\lambda,2m)$-partition-connected.
Let $(V_1,V_2)$ be a bipartition of $V(G)$ with $e_{G}(V_1)+e_{G}(V_2)=bi(G)$.
Let $M=G[V_1]\cup G[V_2]$.
First assume that $bi(G)\le \lambda$ and so $|E(M)|\le \lambda$.
Since $|E(M)|\le \lambda$, it is not difficult to check that we can decompose $G$ into two factors $G_1$ and $G_2$
such that $G_1$ is $(2m-2+2\lambda)$-tree-connected and $G_2$ is $(\lambda, 2m)$-partition-connected and it contains $M$ as well.
 More precisely, one can first decompose $G$ into $2m-2+3\lambda$ spanning trees and a $(0, 2m)$-partition-connected factor and next select $2m-2+2\lambda$ spanning trees excluding the edges $M$ in order to construct the graph $G_1$. 
By applying Lemma~\ref{lem:edge-extension} to the graph $G_2$, there are some edge-disjoint copies of $T$ in $G_2$ covering the edges of $M$. Note that $\delta(G_2)\ge \lambda +2m\ge |E(M)|+2m$. By deleting those copies from $G$ the remaining graph is an $11$-tree-connected bipartite graph of size divisible by $m$ with the bipartition $(V_1,V_2)$. 
Thus by Theorem~\ref{thm:base:bipartite:decomposition}, it admits a $T$-edge-decomposition and so does $G$.

Now, assume that $bi(G)\ge \lambda$.
Since $G$ is $(2m-2+3\lambda)$-tree-connected, by Lemma~\ref{lem:bi:divisible-by-m}, 
every vertex $v$ can be split into two vertices $v_1$ and $v_2$ such that resulting graph $H$ is a bipartite $\lambda$-tree-connected graph with bipartition $(V_1,V_2)$ for which $V_i=\{v_i:v\in V(G)\}$
and for each $v\in V(G)$, $d_{H}(v_1)$ is divisible by $m$.
Thus by Theorem~\ref{thm:bipartite:semiregular}, the bipartite graph $H$
admits a $T$-edge-decomposition and so $G$ admits a $T^*$-edge-decomposition. 
Since $girth(G) > diam(T)$, this $T^*$-edge-decomposition must therefore be a $T$-edge-decomposition. 
Hence the proof is completed.
}\end{proof}
\begin{remark}
{Note that the minimum degree condition in Theorem~\ref{thm:homomorphic:m2} can be dropped and the necessary edge-connectivity in Theorem~\ref{thm:main:88m:2power200m} can also be reduced to $31m$. Moreover, we can develop Theorem~\ref{thm:homomorphic:m2} to the case that $girth(G)=diam(T)$ or $diam(T)\le 4$ in $O(m)$-edge-connected graphs with minimum degree at least $O(m^2)$. The proofs need more complicated arguments. 
We present details in a forthcoming paper.
}\end{remark}
%
%
%
%
%
\section{Graphs with sufficiently large minimum degrees}
In this section, we are going to develop the third item of Theorem~\ref{thm:base:bipartite:decomposition} to arbitrary graphs.
%

\subsection{Method 1: extending subtrees using subgraphs}
Thomassen (2013)~\cite{Thomassen:bistars:2013} and independently Bar{\'a}t and Gerbner (2014) \cite{Barat-Gerbner-2014} showed that it is enough to prove Conjecture~\ref{intro:conj:Barat-Thomassen} for bipartite graphs.
Bar{\'a}t and Gerbner used the following lemma in their proof
which implicitly appeared in the proof of Theorem 2 in \cite{Barat-Gerbner-2014}. 
\begin{lem}{\rm (\cite{Barat-Gerbner-2014})}\label{lem:Barat-Gerbner}
{Let $T$ be a tree of size $m$.
Let $G$ be a simple graph 
having nested bipartite factors $F_{m}\subseteq \cdots \subseteq F_{1}$ with the same bipartition $(X,Y)$.
If $\delta(F_{m})\ge m^3$ and for each $v\in V(F_i)$ satisfying $1\le i\le m-1$,
 $d_{F_{i+1}}(v)\le\frac{1}{m}d_{F_{i}}(v)$,
then there exist edge-disjoint copies of $T$ in $ G[X]\cup F_1$ containing all edges of $G[X]$.
}\end{lem}
We here present a simpler proof for Theorem~\ref{intro:thm:quadratic} based on Lemma~\ref{lem:Barat-Gerbner} but by replacing a factorial upper bound on the necessary minimum degree.
\begin{thm}\label{thm:simple-graph:T-decomposition}
{Let $T$ be a tree of size $m$. If $G$ is a $44m$-edge-connected simple graph of size divisible by $m$ 
and $\delta(G) \ge m^{200m }$, then $G$ admits a $T$-edge-decomposition.
}\end{thm}
\begin{proof}
{We may assume that $m\ge 2$. By Theorem~\ref{thm:edge-connected-partition-connected}, the graph $G$ is 
$(22m,2l)$-partition-connected, where $l=2m 10^{50m}+2(m+1)^{m+2}$.
According to Theorem~\ref{thm:bipartite}, there is a bipartition $(V_1,V_2)$ of $V(G)$ such that $G[V_1,V_2]$ is $(11m,l)$-partition-connected.
Decompose $G[V_1,V_2]$ into a $(11m,2m 10^{50m})$-partition-connected graph $G_0$ 
and two $(0,(m+1)^{m+2})$-partition-connected factors $H_1$ and $H_2$.
By applying lemma~\ref{lem:Werra} with $k=1/(m+1)$ repeatedly, we can find nested factors 
$F_{m, i}\subseteq \cdots \subseteq F_{1, i}$ of $H_i$
such that for each $v\in V(F_s)$ with $1\le s\le m-1$,
 $$(m+1)^{m-s+2}\le \lfloor \frac{1}{m+1} d_{F_s}(v) \rfloor
\le d_{F_{s+1}}(v)\le \lceil \frac{1}{m+1} d_{F_s}(v)\rceil .$$
Note that $\delta(F_m)\ge (m+1)^3$.
In addition, $d_{F_{s+1}}(v) \le \frac{1}{m+1} d_{F_s}(v)+ 1 \le \frac{1}{m} d_{F_s}(v)$, since $\delta(F_s)\ge (m+1)^2$.
Therefore, by Lemma~\ref{lem:Barat-Gerbner}, there exist some edge-disjoint copies of $T$ in $ G[V_i]\cup H_i$ containing all of the edges of $G[V_i]$.
Let $H$ be the factor of $G$ consisting of the edges of $H_1$ and $H_2$ which did not appear in those copies of $T$.
By Theorem~\ref{thm:base:bipartite:decomposition}, the bipartite graph $G_0\cup H$ admits a $T$-edge-decomposition and so does $G$.
}\end{proof}
%
%
%
%
\subsection{Method 2: extending subtrees using isomorphic trees}
In order to improve the lower bound on the minimum degree in Theorem~\ref{thm:simple-graph:T-decomposition} to an exponential bound, we need to recall Lemma 10 in \cite{Bensmail-Harutyunyan-Le-Merker-Thomasse-2017} for the special case $\varepsilon=1/2$ as follows.
\begin{lem}{\rm (\cite{Bensmail-Harutyunyan-Le-Merker-Thomasse-2017})} \label{lem:Lemma10}
{Let $T$ be a tree of size $m$ and let $\delta$ be a real number with $0< \delta < 1$.
Let $G$ be a bipartite simple graph of size divisible by $m$.
If $G$ admits a $T^*$-edge-decomposition $\mathcal{A}^*$ such that for any pair $v\in V(G)$ and $t\in V(T)$, 
$|{\mathcal{A}}(v,t)|\ge (10m)^{18}2^6 \delta ^{-6}$, 
then $G$ contains a factor $H$ having a $T$-edge-decomposition $\mathcal{A}$ satisfying the following properties:
\begin{enumerate}{
\item [$\bullet$] 
For any pair $v\in V(G)$ and $t\in V(T)$, $|\mathcal{A}(v,t)|\ge \frac{1}{2} |\mathcal{A}^*(v,t)|$.
\item [$\bullet$] 
For any triple $v, w\in V(G)$ and $t\in V(T)$ with $v\neq w$, $|\mathcal{A}(v,t) \cap \mathcal{A}(w)| \le \delta |\mathcal{A}(v,t)|$.
}\end{enumerate}
}\end{lem}
Now, we are in a position to prove the first part of Theorem~\ref{intro:thm:quadratic} by giving explicit upper bounds on the necessary edge-connectivity and minimum degree.
\begin{thm}\label{thm:main:88m:2power200m}
{Let $T$ be a tree of size $m$. If $G$ is a $88m$-edge-connected simple graph of size divisible by $m$ 
and $\delta(G) \ge 2^{200m}$, then $G$ admits a $T$-edge-decomposition.
}\end{thm}
\begin{proof}
{We may assume that $m\ge 2$. Note that $\delta(G) \ge 2^{200m} \ge 32ml_0+16ml+88m+4$, where $l_0=(10m)^{18}2^6(2m^3)^6$ and $l=10^{50m}$.
By Theorem~\ref{thm:edge-connected-partition-connected}, the graph $G$ is 
$(44m,16ml_0+8ml+2)$-partition-connected.
According to Theorem~\ref{thm:bipartite}, 
there is a bipartition $(V_1,V_2)$ of $V(G)$ such that $G[V_1,V_2]$ is $(22m , 8ml_0+4ml+1)$-partition-connected. 
We decompose $G[V_1,V_2]$ into an $(11m, 8ml_0+1)$-partition-connected factor $G_1$
 an $(11m,4ml)$-partition-connected factor $G_2$.
We may assume that $G_1$ is $(11m, 8ml_0)$-partition-connected and its size is divisible by $m$ 
by transferring at most $m$ edges from $G_1$ to $G_2$ (if necessary).
Thus by Theorem~\ref{thm:base:bipartite:decomposition}, the graph $G_1$ admits a $T^*$-edge-decomposition $\mathcal{A}^*$ such that for any pair $v\in V(G)$ and $t\in V(T)$, $|\mathcal{A^*}(v,t)|\ge 2l_0$.
By applying Lemma~\ref{lem:Lemma10} with setting $\delta=1/2m^3$, the graph 
$G_1$ contains a factor $G'_1$ having a $T$-edge-decomposition $\mathcal{A}$ such that
 for any pair $v\in V(G)$ and $t\in V(T)$, $|\mathcal{A}(v,t)|\ge \frac{1}{2} |\mathcal{A}^*(v,t)|\ge l_0$, and
for any triple $v, w\in V(G)$ and $t\in V(T)$, $|\mathcal{A}(w) \cap \mathcal{A}(v,t)| \le \frac{1}{2m^3} |\mathcal{A}(v,t)|$.
For a copy $T'$ of a subtree of $T$ and $v\in V(T')$, we denote below by $t(T', v)$ the preimage of $v$. Note that $t(T', v)\in V(T)$ while
$v\in V(T')$.

First, we arbitrarily delete copies of $T$ from $G[V_1]\cup G[V_2]$ as long as possible. 
Let $t_0$ be a fixed internal vertex of $T$.
Next, in the $i$-th step, we select a copy $\mathcal{T}_i$ of a subtree of $T$ including $t_0$
 in the remaining factor of $G[V_1]\cup G[V_2]$ with the maximum size and delete it from $G$.
We say that a vertex in $\mathcal{T}_i$ is {\it unsaturated} 
if its degree in $\mathcal{T}_i$ is less than the degree of its preimage in $T$.
Let $\mathcal{T}_j$ be the first subtree including $v$ as an unsaturated vertex.
If $\mathcal{T}_{j'}$ is another subtree including $v$ as an unsaturated vertex and $vv_{j'}\in E(\mathcal{T}_{j'})$,
 then according to the maximality of $\mathcal{T}_j$, we must have $v_{j'}\in V(\mathcal{T}_j)$. 
Thus there are at most $m$ (more precisely, at most $V(\mathcal{T}_j)$) subtrees including $v$ as an unsaturated vertex.

Let $S$ be the set of all pairs $(\mathcal{T}, v)$ such that $v$ is an unsaturated vertex in a removed subtree $\mathcal{T}$.
Define $P$ to be the bipartite simple graph with $V(P)=S\cup \mathcal{A}$ 
such that every $(\mathcal{T},v)\in S$ is adjacent to all elements in $\mathcal{A}(v, t(\mathcal{T}, v))\subseteq \mathcal{A}$.
Since every vertex is unsaturated in at most $m$ removed subtrees $\mathcal{T}$, every 
$T'$ in $\mathcal{A}$ is adjacent with at most $m|I(T)|$ elements in $S$ in the bipartite graph $P$.
Thus by Lemma~\ref{lem:Werra}, there is a factor $M$ in $P$ such that for all $s\in S$, 
$d_{M}(s)\ge \lfloor \frac{1}{m|I(T)|}d_{P}(s)\rfloor$
and for all $a\in \mathcal{A}$, $d_{M}(a)\le \lceil \frac{1}{m|I(T)|}d_{P}(a)\rceil \le 1$.
This means that for every $(\mathcal{T},v)\in S$, there is a subset $\mathcal{B}(\mathcal{T},v)$ of 
$\mathcal{A}(v, t(\mathcal{T},v))$ 
such that $|\mathcal{B}(\mathcal{T},v)|\ge \frac{1}{m|I(T)|}|\mathcal{A}(v, t(\mathcal{T},v))|-1$ and all of $\mathcal{B}(\mathcal{T},v)$ are disjoint. 
We are going to extend every removed subtree $\mathcal{T}$ to a copy of $T$ using the edges of trees in sets $\mathcal{B}(\mathcal{T},v)$, where $v$ is an unsaturated vertex in $\mathcal{T}$.
Let $u_1,\ldots, u_n$ be unsaturated vertices of $\mathcal{T}$.
We construct inductively nested trees $\mathcal{T}_i$ which are isomorphic to copies of subtrees of $T$.
In addition, $v_i,\ldots, v_n$ are unsaturated in them.
First, we set $\mathcal{T}_0=\mathcal{T}$. Let $i\in \{0,\ldots, n\}$.
For notational simplicity, let us write $t=t(\mathcal{T},v_{i})$. Thus
$$|\mathcal{B}(\mathcal{T}, v_{i})\setminus \cup_{w\in V(\mathcal{T}_i)\setminus \{v_i\}} ( \mathcal{A}(v_{i},t) \cap \mathcal{A}(w))|
\ge 
\frac{1}{m|I(T)|}|\mathcal{A}(v_{i}, t)|-1 -\frac{m}{2m^3} |\mathcal{A}(v_{i},t)|\ge \frac{1}{2m^2}|\mathcal{A}(v_{i},t)| -1>0.$$
Thus there is a tree $T_i\in \mathcal{B}(\mathcal{T}, v_{i})$ such that $V(T_i)\cap V(\mathcal{T}_i)=\{v_{i}\}$.
We here denote by $T(t_0, t)$ the component of $T-e$ including $t$, where $e$ is the unique edge incident with $t$ lying in the path of $T$ connecting $t_0$ and $t$.
Let $T'_i$ be the subtree of $T_i$ consisting of the images of the edges of $T(t_0, t)$. Now, we define $\mathcal{T}_{i}=\mathcal{T}_{i-1}\cup T'_i$.
By repeating this process, we obtain a tree $\mathcal{T}_{n}$ which is isomorphic to $T$ and it includes $\mathcal{T}$.
Since all of $\mathcal{B}(\mathcal{T},v)$ are disjoint, all of such trees $\mathcal{T}_{n}$ are edge-disjoint. 

In other words, there are some edge-disjoint copies of $T$ in $ G[V_1]\cup G[V_2]\cup G_1$ covering all of the edges of $G[V_1]\cup G[V_2]$. If we delete these copies, the resulting graph $G'$ contains $G_2$ and its size must be divisible by $m$.
So, $G'$ is $(11m, 4m10^{50m})$-partition-connected. Therefore, by Theorem~\ref{thm:base:bipartite:decomposition}, the bipartite graph $G'$ admits a $T$-edge-decomposition and so does $G$. Hence the proof is completed.
}\end{proof}
%
%
%
%
%
%
%
%
%
%

\end{document}